\documentclass{siamltex}

\usepackage{amsmath, amssymb}
\usepackage[dvips]{graphicx}
\usepackage{algorithm}
\usepackage{algorithmic}
\usepackage{url}
\usepackage{footnote}
\usepackage{multirow}
\usepackage{color}

\newtheorem{assumption}{Assumption}[section]

\title{Accelerating Nesterov's Method for Strongly Convex Functions with
  Lipschitz Gradient}

\author{Xiangrui Meng\thanks{ICME, Stanford University, Stanford, CA 94305 ({\tt
      mengxr@stanford.edu})} \and Hao Chen\thanks{Department of Statistics,
    Stanford University, Stanford, CA 94305 ({\tt haochen@stanford.edu})}}

\begin{document}
 
\maketitle

\begin{abstract}
  We modify Nesterov's constant step gradient method for strongly convex
  functions with Lipschitz continuous gradient described in Nesterov's
  book. Nesterov shows that $f(x_k) - f^* \leq L \prod_{i=1}^k ( 1 - \alpha_k )
  \| x_0 - x^* \|_2^2$ with $\alpha_k = \sqrt{\rho}$ for all $k$, where $L$ is
  the Lipschitz gradient constant and $\rho$ is the reciprocal condition number
  of $f(x)$. Hence the convergence rate is $1-\sqrt{\rho}$. In this work, we try
  to accelerate Nesterov's method by adaptively searching for an $\alpha_k >
  \sqrt{\rho}$ at each iteration. The proposed method evaluates the gradient
  function at most twice per iteration and has some extra Level 1 BLAS
  operations. Theoretically, in the worst case, it takes the same number of
  iterations as Nesterov's method does but doubles the gradient calls. However,
  in practice, the proposed method effectively accelerates the speed of
  convergence for many problems including a smoothed basis pursuit denoising
  problem.
\end{abstract}




\section{Introduction}
\label{sec:intro}

First-order methods for convex optimization have drawn great interest in recent
years as the problem scale goes larger and larger. High-order methods do not fit
the scene quite well because they generally need more memory than first-order
methods and take many more operations per iteration. However, the slow
convergence rate of first-order methods prevents them from practical use. For
example, the constant step gradient descent method converges at the speed of
$\mathcal{O}(1/k)$ for functions with Lipschitz gradient (with constant $L$),
where $k$ is the number of iterations.  It means that we need one million
iterations to reach $f(x_k) - f^* < \mathcal{O}(10^{-6}) ( f(x_0) - f^*
)$. Nesterov~\cite{nesterov1983method} advanced the field with a first-order
method converging at the speed of $\mathcal{O}(1/k^2)$. We refer to this method
as $\mathcal{N}_L$. To reach the same precision as in the previous example,
$\mathcal{N}_L$ only needs one thousand iterations. Nesterov not only shows the
method is faster than the gradient descent method but also shows that it is
optimal among all first-order methods on functions with Lipschitz gradient. To
seek a first-order method with higher-order convergence, we have to restrict the
functions of interest. Nesterov~\cite{nesterov2004introductory} considered
functions with both Lipschitz gradient and strong convexity (with parameter
$\mu$), and he constructed another first-order method with linear convergence
rate, referred to as $\mathcal{N}_{\mu,L}$. The gradient descent method can also
achieve linear convergence on those functions. Nevertheless, to reach a given
precision, the number of iterations the gradient descent method needs is
$\mathcal{O}(\kappa)$, where $\kappa$ is the condition number of the objective
function, while the number of iterations $\mathcal{N}_{\mu,L}$ needs is only
$\mathcal{O}(\sqrt{\kappa})$, which is proved to be optimal too.

In this work, we are interested in accelerating $\mathcal{N}_{\mu,L}$ in a
practical way. In section \ref{sec:nesterov}, we briefly review how Nesterov
constructs $\mathcal{N}_{\mu,L}$. Then we present our modification to Nesterov's
method in section \ref{sec:nesterov-alpha}. Related work on improving Nesterov's
methods is discussed in section \ref{sec:related-work}, and section
\ref{sec:numer-exper} reports numerical results.



\section{Nesterov's method}
\label{sec:nesterov}

We briefly review Nesterov's constant step gradient method for strongly convex
functions with Lipschitz gradient, referred to as $\mathcal{N}_{\mu,L}$, and its
convergence properties.  The content is mostly taken from
Nesterov~\cite{nesterov2004introductory} with some simplifications. We keep this
section short and concise but detail how Nesterov constructs the method because
our modification is based on it. We begin with the definition of $S_{\mu,L}$,
the class of strongly convex functions with Lipschitz gradient, and an
assumption on first-order methods.

\begin{definition}
  A continuous differentiable function $f(x)$ is in $S_{\mu,L}(\Omega)$ for some
  $L \geq \mu > 0$ if for any $x, y \in \Omega$ we have both of the following:
  \begin{equation}
    \label{eq:lip-grad}
    \| f'(x) - f'(y) \|_2 \leq L \| x - y \|_2,
  \end{equation}
  \begin{equation}
    \label{eq:scvx}
    f(y) \geq f(x) + \langle f'(x), y - x \rangle
    + \frac{\mu}{2} \| y - x \|_2^2.
  \end{equation} 
  The value $\kappa = L/\mu$ is called the condition number of $f(x)$ and $\rho
  = 1/\kappa$ is called the reciprocal condition number of $f(x)$.
\end{definition}

Throughout, we assume that for a function from $S_{\mu,L}$ either $\mu$ and $L$
or a lower bound of $\mu$ and an upper bound of $L$ are given.

\begin{assumption}\textnormal{\cite[p.\,59]{nesterov2004introductory}}
  \label{asp:first_order}
  A first-order method generates a sequence of points $\{x_k\}$ such that $x_k
  \in x_0 + \text{Span}\left\{ f'(x_0), \ldots, f'(x_{k-1}) \right\}, k \geq 1.$
\end{assumption}

For functions in $S_{\mu,L}$, Nesterov constructs a first-order method,
$\mathcal{N}_{\mu,L}$, and shows that it matches a lower complexity bound for
first-order methods satisfying Assumption \ref{asp:first_order} up to a constant
factor in the sense of worst-case number of
iterations. Nesterov~\cite{nesterov2004introductory} gives more details on the
optimality. Note that Assumption \ref{asp:first_order} is very mild, as most
first-order methods fall into the framework, which secures the optimality of
$\mathcal{N}_{\mu,L}$. To construct such an optimal first-order method, Nesterov
introduces an \emph{estimate sequence} and shows how it helps derive
$\mathcal{N}_{\mu,L}$ and prove its convergence rate.
\begin{definition}\textnormal{\cite[p.\,72]{nesterov2004introductory}}
  \label{def:2.2.1}
  A pair of sequences $\{\phi_k(x)\}$ and $\{\lambda_k\}$, $\lambda_k \geq 0$ is
  called an estimate sequence of $f(x)$ if $\lambda_k \to 0$ and we have
  \begin{equation}
    \label{eq:3}
    \phi_k(x) \leq ( 1 - \lambda_k) f(x) + \lambda_k \phi_0(x), \quad \forall x \in \mathbb{R}^n \text{ and } k\geq 0.
  \end{equation}
\end{definition}
\begin{lemma}\textnormal{\cite[p.\,72]{nesterov2004introductory}}
  \label{lemma:2.2.1}
  If the pair of sequences $\{\phi_k(x)\}$ and $\{\lambda_k\}$ is an estimate
  sequence of $f(x)$ and for some sequence $\{ x_k \}$ we have
  \begin{equation}
    \label{eq:2.2.2}
    f(x_k) \leq \phi_k^* \equiv \min_{x \in \mathbb{R}^n} \phi_k(x),
  \end{equation}
  then $f(x_k) - f^* \leq \lambda_k [ \phi_0(x^*) - f^* ] \to 0$, where $x^*$ is
  the optimal value of $f(x)$.
\end{lemma}

Now the question becomes, given $f(x) \in S_{\mu,L}$, how can we construct an
estimate sequence of $f(x)$ and generate a sequence $\{ x_k \}$ satisfying
\eqref{eq:2.2.2}. To construct an estimate sequence, we have the following
lemma.
\begin{lemma}\textnormal{\cite[p.\,72]{nesterov2004introductory}}
  \label{lemma:2.2.2}
  Assume the following:
  \begin{enumerate}
  \item $f \in S_{\mu,L}( \mathbb{R}^n )$,
  \item $\phi_0(x)$ is an arbitrary function on $\mathbb{R}^n$,
  \item $\{y_k\}$ is an arbitrary sequence in $\mathbb{R}^n$,
  \item $\{ \alpha_k \}: \alpha_k \in (0,1)$, $\sum_{k=0}^{\infty} \alpha_k
    = \infty$,
  \item $\lambda_0 = 1$.
  \end{enumerate}
  Then the pair of sequences $\{\phi_k(x)\}$, $\{\lambda_k\}$ recursively
  defined by
  \begin{eqnarray}
    \label{eq:4}
    \lambda_{k+1} &=& ( 1 - \alpha_k ) \lambda_k, \\
    \phi_{k+1}(x) &=& ( 1 - \alpha_k ) \phi_k(x) + \alpha_k \left[ f(y_k) 
      + \langle f'(y_k), x - y_k \rangle 
      + \frac{\mu}{2} \| x - y_k \|_2^2 \right], 
  \end{eqnarray}
  is an estimate sequence.
\end{lemma}

We see that Lemma \ref{lemma:2.2.2} leaves us freedom in the choice of
$\phi_0(x)$, $\{y_k\}$, and $\{\alpha_k\}$. To combine the result from Lemma
\ref{lemma:2.2.1}, we should choose a simple $\phi_0(x)$ such that $\phi_k^*$ is
easy to obtain in explicit form, and choose $\{y_k\}$ and $\{\alpha_k\}$
appropriately such that we can find $x_k$ satisfying $f(x_k) \leq \phi_k^*$ for
each $k$. The following lemma is a simplified version of Lemma $2.2.3$ of
Nesterov~\cite[p.\,69]{nesterov2004introductory}.
\begin{lemma} 
  \label{lemma:2.2.3}
  Let $\phi_0(x) = \phi_0^* + \frac{\mu}{2} \| x - v_0 \|_2^2$. Then the process
  defined in Lemma \ref{lemma:2.2.2} preserves the canonical form of functions
  $\{ \phi_k(x) \}$:
  \begin{equation}
    \label{eq:1}
    \phi_k(x) \equiv \phi_k^* + \frac{\mu}{2} \| x - v_k \|_2^2,
  \end{equation}
  where the sequences $\{v_k\}$ and $\{\phi_k^*\}$ are defined as follows:
  \begin{eqnarray}
    \label{eq:v}
    v_{k+1} &=& ( 1 - \alpha_k ) v_k + \alpha_k y_k 
    - \frac{\alpha_k}{\mu} f'(y_k), \\
    \label{eq:phi}
    \phi_{k+1}^* &=& ( 1 - \alpha_k ) \phi_k^* + \alpha_k f(y_k) 
    - \frac{\alpha_k^2}{2 \mu} \| f'(y_k) \|_2^2 \\
    \nonumber && \null + \alpha_k ( 1 - \alpha_k ) \left( 
      \frac{\mu}{2} \| y_k - v_k \|_2^2 
      + \langle f'(y_k), v_k - y_k \rangle \right).
  \end{eqnarray}
\end{lemma}
Suppose we have $\phi_k^* \geq f(x_k)$ at the $k$-th iteration. By
\eqref{eq:scvx} we know
\begin{equation*}
  \phi_k^* \geq f(y_k) + \langle f'(y_k), x_k - y_k \rangle 
  + \frac{\mu}{2} \| x_k - y_k \|_2^2.
\end{equation*}
Plugging it into \eqref{eq:phi}, we get
\begin{eqnarray}
  \label{eq:suff_orig}
  \phi_{k+1}^* &\geq& f(y_k) - \frac{\alpha_k^2}{2 \mu} \| f'(y_k) \|^2 
  + ( 1 - \alpha_k ) \langle f'(y_k), \alpha_k ( v_k - y_k ) 
  + ( x_k - y_k ) \rangle \\
  \nonumber
  && \null + \frac{\mu ( 1 - \alpha_k )}{2} \left( 
    \alpha_k \| v_k - y_k \|_2^2 + \| x_k - y_k \|_2^2 \right).
\end{eqnarray}
Remember that $y_k$ is arbitrary. We can choose $y_k = ( x_k + \alpha_k v_k ) /
( 1 + \alpha_k )$ to eliminate the linear term associated with $f'(y_k)$ and
drop the sum of squares. Then we have
\begin{equation*}
  \phi_{k+1}^* \geq  f(y_k) - \frac{\alpha_k^2}{2 \mu} \| f'(y_k) \|_2^2.
\end{equation*}
Therefore, to make $\phi_{k+1}^* \geq f(x_{k+1})$, it is sufficient to find an
$x_{k+1}$ such that
\begin{equation*}
  f(x_{k+1}) \leq f(y_k) - \frac{\alpha_k^2}{2 \mu} \| f'(y_k) \|_2^2.
\end{equation*}
Because $f'(x)$ is Lipschitz continuous with constant $L$, by choosing $x_{k+1}
= y_k - \frac{1}{L} f'(y_k)$ we can always ensure
\begin{equation}
  \label{eq:x_k1}
  f(x_{k+1}) \leq f(y_k) - \frac{1}{2L}\|f'(y_k)\|_2^2.
\end{equation}
Comparing the two inequalities above, we see setting $\alpha_k = \sqrt{\mu/L} =
\sqrt{\rho}$ would suffice. Now we can further simplify the update scheme by
knocking out $\{v_k\}$. We have
\begin{eqnarray*}
  y_{k+1} &=& \frac{ x_{k+1} + \alpha v_{k+1} } { 1 + \alpha } = 
  \frac{ x_{k+1} + \alpha \left[ ( 1 - \alpha ) v_k 
      + \alpha y_k - \frac{\alpha}{\mu} f'(y_k) \right] }{ 1 + \alpha } \\
  &=& \frac{ x_{k+1} + \alpha \left[ ( 1 - \alpha ) \frac{(1+\alpha) y_k 
        - x_k}{\alpha}  + \alpha y_k 
      - \frac{\alpha}{\mu} f'(y_k) \right] }{ 1 + \alpha } \\
  &=& x_{k+1} + \frac{1-\alpha}{1+\alpha} ( x_{k+1} - x_{k} ).
\end{eqnarray*}
We summarize this method in Algorithm \ref{alg:nesterov_scvx_3}, which is
extremely simple. The term $\frac{1-\sqrt{\rho}}{1+\sqrt{\rho}}$ is called the
acceleration parameter.
\begin{algorithm}
  \caption{$\mathcal{N}_{\mu,L}$, Nesterov's constant step scheme,
    III \cite[p.\,81]{nesterov2004introductory}}
  \label{alg:nesterov_scvx_3}
  \begin{algorithmic}[1]
    \STATE Given $f(x) \in S_{\mu,L}$ and $x_0$, set $\rho = \mu/L$ and $y_0 =
    x_0$.
    
    \FOR{$k=0,1,\ldots$ until convergence}
    
    \STATE $x_{k+1} = y_k - \frac{1}{L} f'(y_k)$
    
    \STATE $y_{k+1} = x_{k+1} + \frac{1-\sqrt{\rho}}{1+\sqrt{\rho}} (
    x_{k+1} - x_k )$

    \ENDFOR
  \end{algorithmic}
\end{algorithm}

Let $ \phi_0^* = f(x_0)$, then Lemmas \ref{lemma:2.2.1} and \ref{lemma:2.2.2}
characterize the convergence of $\mathcal{N}_{\mu,L}$.  The following theorem is
a simplified version of Theorem 2.2.3 of
Nesterov~\cite[p.\,80]{nesterov2004introductory}:
\begin{theorem}
  \label{thm:convergence}
  $\mathcal{N}_{\mu,L}$ (Algorithm \ref{alg:nesterov_scvx_3}) generates a
  sequence $\{ x_k \}$ such that
  \begin{equation}
    \label{eq:5}
    f(x_k) - f^* \leq \left( 1 - \sqrt{\rho} \right)^k \left( f(x_0) 
      + \frac{\mu}{2} \| x_0 - x^* \|_2^2 - f^* \right) 
    \leq L ( 1 - \sqrt{\rho} )^k \| x_0 - x^* \|_2^2.
  \end{equation}
\end{theorem}

Note that Nesterov actually provides three variants in
\cite{nesterov2004introductory} and what we mentioned here is the third one. For
the other two, $\{\alpha_k\}$ is not a constant sequence but deterministic and
having $\alpha_k \to \sqrt{\rho}$ as $k \to \infty$; hence the asymptotic
convergence rate is still $1-\sqrt{\rho}$. In practice, they perform quite
similarly, while the third is the least expensive among the three variants.


\section{Accelerating Nesterov's method with adaptive $\alpha_k$}
\label{sec:nesterov-alpha}

In $\mathcal{N}_{\mu,L}$, the rate of decrease of $f(x_k)-f^*$ at the $k$-th
iteration is bounded by $1-\alpha_k$, where $\alpha_k = \sqrt{\rho}$ for all
$k$. Our modified method is based on the following idea: trying to make
$\alpha_k$ larger than $\sqrt{\rho}$ at each iteration in order to accelerate
the convergence. To see how it works, we need to revisit Nesterov's
construction, particularly the inequality \eqref{eq:suff_orig}. Given
\eqref{eq:suff_orig}, it is sufficient to find $\alpha_k$, $y_k$, and $x_{k+1}$
such that
\begin{eqnarray*}
  f(x_{k+1}) &\leq& f(y_k) - \frac{\alpha_k^2}{2 \mu} \| f'(y_k) \|_2^2 \\
  && \null + ( 1 - \alpha_k ) \langle f'(y_k), \alpha_k ( v_k - y_k ) 
  + ( x_k - y_k ) \rangle \\
  && \null + \frac{\mu ( 1 - \alpha_k )}{2} 
  \left( \alpha_k \| v_k - y_k \|_2^2 + \| x_k - y_k \|_2^2 \right)
\end{eqnarray*}
to retain \eqref{eq:2.2.2}: $f(x_{k+1}) \leq \phi_{k+1}^*$. The goal of finding
an $\alpha_k\in[0,1]$ as large as possible leads us to the following
optimization problem:
\begin{eqnarray}
  \nonumber \text{maximize} &\quad& \alpha_k \in [0,1] \\
  \label{opt:max_alpha}
  \text{subject to} && f(x_{k+1}) \leq f(y_k) 
  - \frac{\alpha_k^2}{2 \mu} \| f'(y_k) \|_2^2 \\
  \nonumber && \null + ( 1 - \alpha_k ) \langle f'(y_k), \alpha_k ( v_k - y_k ) 
  + ( x_k - y_k ) \rangle \\
  \nonumber && \null + \frac{\mu ( 1 - \alpha_k )}{2} 
  \left( \alpha_k \| v_k - y_k \|_2^2 + \| x_k - y_k \|_2^2 \right),
\end{eqnarray}
where $\alpha_k$, $y_k$, and $x_{k+1}$ are free variables, while $v_k$ is
determined at step $k$. Apparently, one optimal solution is given by $\alpha_k^*
= 1$ and $x_{k+1}^* = y_k^* = x^*$. However, $x^*$ is unknown and $y_k$ and
$x_{k+1}$ should be derived from past iterates and gradients. So we can only
expect a sub-optimal solution that is good and easy to obtain. To restrict the
optimization problem, we fix the choices of $y_k$ and $x_{k+1}$, following
Nesterov:
\begin{eqnarray}
  \label{eq:y_k}
  y_k =  \frac{ x_k + \alpha_k v_k }{ 1 + \alpha_k }, 
  \quad x_{k+1} = y_k - \frac{1}{L} f'(y_k).
\end{eqnarray}
The choice of $y_k$ eliminates the linear term associated with $f'(y_k)$ and we
have
\begin{equation*}
  \frac{\mu ( 1 - \alpha_k )}{2} \left( \alpha_k \| v_k - y_k \|_2^2 
    + \| x_k - y_k \|_2^2 \right) 
  = \frac{\mu \alpha_k ( 1 - \alpha_k )}{2 ( 1+ \alpha_k ) } \| x_k - v_k \|_2^2.
\end{equation*}
Plugging \eqref{eq:y_k} into \eqref{opt:max_alpha}, we get
\begin{eqnarray}
  \nonumber \text{maximize} &\quad& \alpha_k \in [0,1] \\
  \label{opt:max_alpha_yk}
  \text{subject to} && f(x_{k+1}) \leq f(y_k) - \frac{\alpha_k^2}{2 \mu} 
  \| f'(y_k) \|_2^2 + \frac{\mu \alpha_k ( 1 - \alpha_k )}{2 ( 1+ \alpha_k ) } 
  \| x_k - v_k \|_2^2.
\end{eqnarray}
Since evaluating the function costs time, we would be better to eliminate
$f(x_{k+1})$ and $f(y_k)$ from the above inequality.  Note that $f'(x)$ is
Lipschitz continuous and hence the choice of $x_{k+1}$ implies \eqref{eq:x_k1}.
Reinforcing the inequality \eqref{opt:max_alpha_yk} by \eqref{eq:x_k1}, we get
\begin{eqnarray}
  \label{opt:max_alpha_relax}
  \nonumber \text{maximize} &\quad& \alpha_k \\
  \text{subject to} && \left( \alpha_k^2 - \rho \right) 
  \left\| f' \left( \frac{x_k + \alpha_k v_k}{ 1 + \alpha_k} 
    \right) \right\|_2^2 \leq  \mu^2 \| x_k - v_k \|_2^2 
  \frac{\alpha_k ( 1 - \alpha_k )}{1+ \alpha_k},
\end{eqnarray}
where $\alpha_k \in [0,1]$ is implied by the constraint. Now $\alpha_k$ is the
only free variable. The constraint always holds if $\alpha_k =
\sqrt{\rho}$. Moreover, the constraint is not tight at $\alpha_k = \sqrt{\rho}$
if $x_k \neq v_k$, which is generally the case. So we can almost always expect
an $\alpha_k > \sqrt{\rho}$ at each iteration. However, the problem is still
nonlinear and solving it may lead to many function calls to the gradient
function, which is inefficient because with those gradient calls we can proceed
with the same number of iterations in $\mathcal{N}_{\mu,L}$. We try to solve
this problem approximately with the hope of getting $\alpha_k$ as large as
possible in one or two gradient calls. The idea is inspired by the following
lemma.
\begin{lemma}\label{lemma:v_k}
  Given $f(x) \in S_{\mu,L}$, let the pair of sequences $\{\phi_k(x) = \phi_k^*
  + \frac{\mu}{2} \| x - v_k \|_2^2 \}$ and $\{\lambda_k\}$ be as defined in
  Lemmas \ref{lemma:2.2.2} and \ref{lemma:2.2.3}. If for some sequence $\{ x_k
  \}$ we have $f(x_k) \leq \phi_k^*$ for all $k$, then $\lim_{k \to \infty} v_k
  = x^*$.
\end{lemma}
\begin{proof}
  The pair of sequences $\{\phi_k(x)\}$ and $\{ \lambda_k \}$ is an estimate
  sequence. By definition we have 
  \begin{equation*}
    \phi_k(x^*) \leq ( 1 - \lambda_k ) f(x^*)+ \lambda_k \phi_0(x^*), 
    \quad \forall k \geq 0,
  \end{equation*}
  and $\lim_{k \to \infty} \lambda_k = 0$.  Given $f(x^*) \leq f(x_k) \leq
  \phi_k^*$, we know
  \begin{eqnarray*}
    \frac{\mu}{2} \| v_k - x^* \|_2^2 = \phi_k(x^*) 
    - \phi_k^* \leq ( 1 - \lambda_k ) f( x^* ) + \lambda_k \phi_0(x^*) - f(x^*) 
    = \lambda_k( \phi_0(x^*) - f(x^*) ).
  \end{eqnarray*}
  Letting $k \to \infty$ on both sides, we have $\lim_{k \to \infty}
  \|v_k-x^*\|_2^2 = 0$ and hence $\lim_{k \to \infty} v_k = x^*$.
\end{proof}

As long as $y_k$ is chosen as in \eqref{eq:y_k}, by Lemmas \ref{lemma:2.2.1} and
\ref{lemma:v_k} we have $ \lim_{k\to\infty} y_k = x^* $ and thus the global
trend for $\| f'(y_k) \|_2$ is decreasing. So if assuming the change between two
contiguous iterations is small, we can use $\| f'(y_{k-1}) \|_2$ as an
approximate upper bound on $\| f'(y_k) \|_2$ to save the cost of evaluating
gradients since $\| f'(y_{k-1})\|_2$ is already calculated in the previous step.
The modified constraint is therefore
\begin{equation}
  \label{eq:y_k1}
  \left( \alpha_k^2 - \rho \right) \left\| f'(y_{k-1}) \right\|_2^2 
  \leq \frac{\mu^2 \alpha_k ( 1 - \alpha_k )}{1+ \alpha_k} \| x_k - v_k \|_2^2,
\end{equation}
which is equivalent to
\begin{equation*}
  \alpha_k^3 + ( 1 + D_k ) \alpha_k^2 - ( \rho + D_k ) \alpha_k - \rho \leq 0,
\end{equation*}
where $D_k = \mu^2 \| x_k - v_k \|_2^2 / \| f'(y_{k-1}) \|_2^2$. Let's consider
how to pick an $\alpha_k$ at each step. Define
\begin{equation*}
  \eta_k( \alpha ) = \alpha^3 + ( 1 + D_k ) \alpha^2 
  - ( \rho + D_k ) \alpha - \rho.
\end{equation*}
If $D_k=0$ ($x_k = v_k$), then the largest $\alpha$ satisfying $\eta_k(\alpha)
\leq 0$ is $\sqrt{\rho}$.  Assume that $D_k > 0$ and $\rho < 1$. It is easy to
verify the following properties of $\eta_k(\alpha)$ by checking its first and
second derivatives:
\begin{itemize}
\item $\eta_k(\alpha_0) < 0$, where $\alpha_0 = \sqrt{\rho}$,
\item $\eta_k(\alpha)$ has exactly one positive local minimum, denoted by
  $\beta_k$,
\item $\eta_k(\alpha)$ has exactly one positive root, denoted by $\gamma_k$.
\end{itemize}
\begin{figure}
  \centering
  \includegraphics[width=0.43\textwidth, clip, trim=1.6cm 0.6cm 1.7cm
  1.3cm]{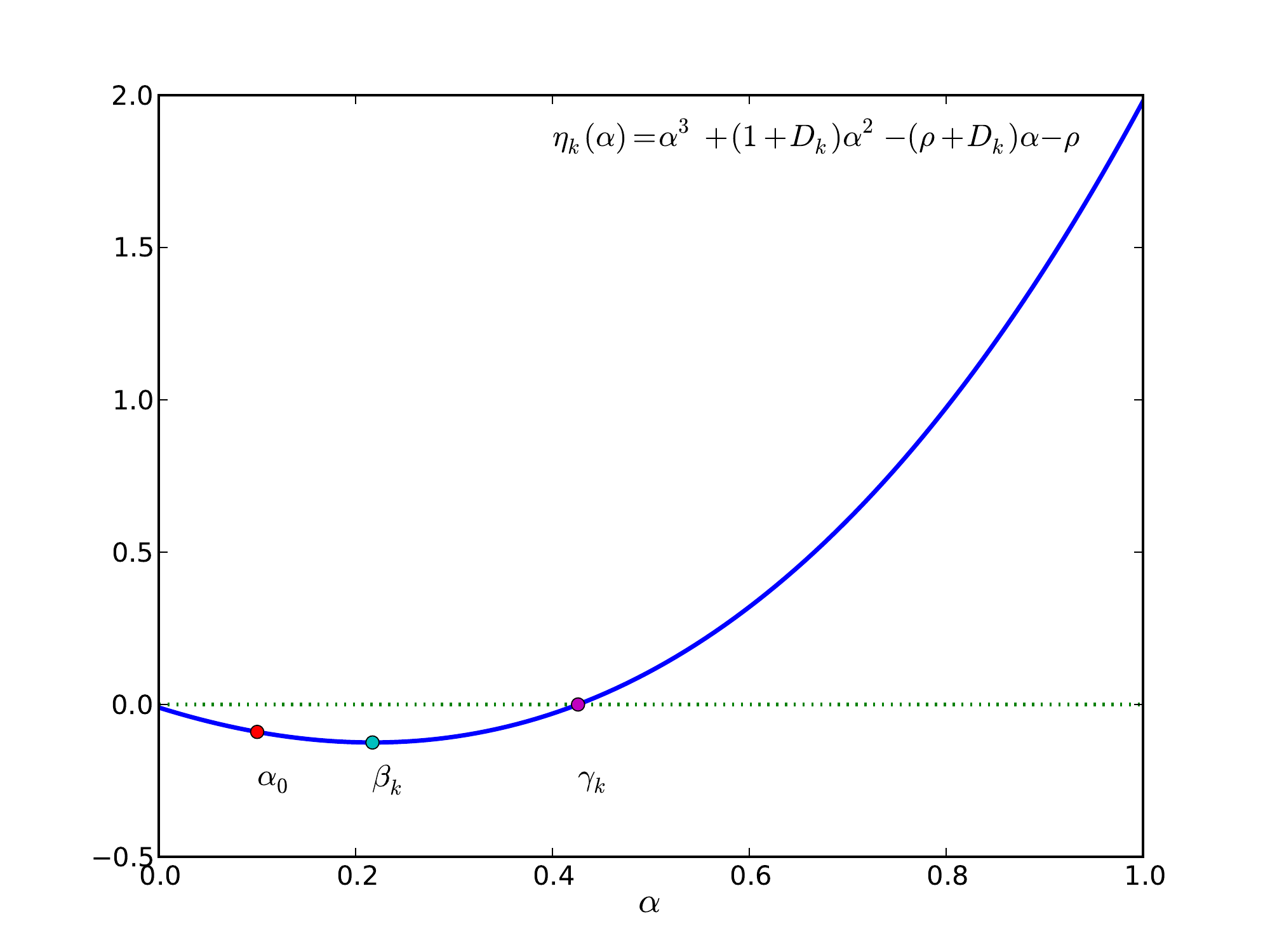}
  \caption{Choosing $\alpha_k$: 1) $\alpha_k = \alpha_0 = \sqrt{\rho}$, we
    always have $\eta_k(\alpha_0) \leq 0$, 2) $\alpha_k = \gamma_k$, the
    positive root of $\eta_k(\alpha)$, is an aggressive choice because we don't
    always have $\|f'(y_{k-1})\|_2 \geq \|f'(y_k)\|_2$, 3) $\alpha_k = \beta_k$,
    the local minimum of $\eta_k(\alpha)$, is a generally safe choice if
    $\beta_k > \alpha_0$.}
  \label{fig:alpha_k}
\end{figure}
Figure \ref{fig:alpha_k} shows a typical plot of $\eta_k(\alpha)$ with
$\alpha_0$, $\beta_k$, and $\gamma_k$. Note that $\beta_k$ is not necessarily
larger than $\alpha_0$. Choosing $\alpha_k = \alpha_0$ always leads to a valid
estimate sequence that guarantees convergence. Given $\alpha_0$ as our fallback
choice, we try to be more aggressive. $\alpha_k = \gamma_k$ is apparently the
most aggressive choice. However, if we choose $\alpha_k = \gamma_k$,
\eqref{opt:max_alpha_relax} may break frequently because $\| f'(y_{k-1})\|_2$ is
not always an upper bound on $\| f'(y_k) \|_2$.  If $\beta_k > \alpha_0$,
$\alpha_k = \beta_k$ may be a safe choice that is more robust to the violation
of $\|f'(y_{k-1})\|_2 \geq \|f'(y_k)\|_2$. Based on these observations, we
propose four heuristics (from conservative to aggressive) to pick an $\alpha_k$
and compare their performance later in section \ref{sec:numer-exper}. They are
as follows:
\begin{enumerate}
\item $\alpha_k = \max( \alpha_0, \beta_k )$,
\item $\alpha_k = \frac{1}{2}( \alpha_0 + \gamma_k )$,
\item $\alpha_k = \frac{1}{2}( \max( \alpha_0, \beta_k ) + \gamma_k )$.
\item $\alpha_k = \gamma_k$.
\end{enumerate}
As mentioned before, having an $\alpha_k$ satisfying constraint \eqref{eq:y_k1}
doesn't imply that $\alpha_k$ is feasible in \eqref{opt:max_alpha_relax}. If our
guess doesn't meet the constraint, we fall back to Nesterov's choice $\alpha_k =
\sqrt{\rho}$ without making extra effort in searching for an $\alpha_k >
\sqrt{\rho}$. Therefore, the modified method calls the gradient function at most
twice per iteration and has at least the same rate of convergence as
$\mathcal{N}_{\mu,L}$ in terms of number of iterations.

We summarize our modified method in Algorithm \ref{alg:nesterov_scvx_alpha} and
refer to it as $\mathcal{N}_{\mu,L}^\alpha$. To differentiate the four
heuristics we proposed to pick an $\alpha_k$, we call the corresponding variants
$\mathcal{N}_{\mu,L}^{\alpha,1}$, $\mathcal{N}_{\mu,L}^{\alpha,2}$,
$\mathcal{N}_{\mu,L}^{\alpha,3}$, and $\mathcal{N}_{\mu,L}^{\alpha,4}$,
respectively.
\begin{algorithm}
  \caption{$\mathcal{N}_{\mu,L}^\alpha$, Nesterov's constant step scheme with
    adaptive $\alpha_k$}
  \label{alg:nesterov_scvx_alpha}
  \begin{algorithmic}[1]

    \STATE Given $f(x) \in S_{\mu,L}(\mathbb{R}^n)$ and $x_0$, set $v_0 = y_0 =
    x_0$ and $\alpha_0 = \sqrt{\rho} = \sqrt{\mu/L}$.

    \STATE Compute $x_1 = y_0 - \frac{1}{L} f'(y_0)$.

    \FOR{$k=1,2,\ldots$ until convergence}
    
    \STATE Compute $v_{k} = ( 1 - \alpha_{k-1} ) v_{k-1} + \alpha_{k-1} y_{k-1} -
    \frac{\alpha_{k-1}}{\mu} f'(y_{k-1})$.

    \STATE Let $D_k = \mu^2 \| x_k - v_k \|_2^2 / \| f'(y_{k-1}) \|_2^2$. Choose
    an $\tilde{\alpha}_k \geq \sqrt{\rho}$ such that
    \begin{equation*}
      \tilde{\alpha}_k^3 + ( 1 + D_k ) \tilde{\alpha}_k^2 
      - ( \rho + D_k ) \tilde{\alpha}_k - \rho \leq 0.
    \end{equation*}

    \STATE Compute $\tilde{y}_k = ( x_k + \tilde{\alpha}_k v_k ) / ( 1 +
    \tilde{\alpha}_k )$ and $\tilde{x}_{k+1} = \tilde{y}_k - \frac{1}{L}
    f'(\tilde{y}_k)$.
    
    \STATE Validate whether we have $f(\tilde{x}_{k+1}) \leq \phi_{k+1}^*$ by
    verifying a more stringent inequality
    \begin{equation*}
      (\tilde{\alpha}_k^2 - \rho) \| f'( \tilde{y}_k ) \|_2^2 \leq \mu^2 \| x_k - v_k \|_2^2 \frac{\tilde{\alpha}_k ( 1 - \tilde{\alpha}_k )}{1 + \tilde{\alpha}_k}.
    \end{equation*}
    
    \STATE If valid, let $\alpha_k = \tilde{\alpha}_k$, $y_k = \tilde{y_k}$, and
    $x_{k+1}
    = \tilde{x}_{k+1}$.\\
    Otherwise, let $\alpha_k = \sqrt{\rho}$, $y_k = ( x_k + \alpha_k v_k ) / ( 1
    + \alpha_k )$, and $x_{k+1} = y_k - \frac{1}{L} f'(y_k)$.

    \ENDFOR
  \end{algorithmic}
\end{algorithm}


\section{Related work}
\label{sec:related-work}

In this section, we discuss related work on accelerating Nesterov's methods
$\mathcal{N}_L$ and $\mathcal{N}_{\mu,L}$. In both $\mathcal{N}_L$ and
$\mathcal{N}_{\mu,L}$, the global Lipschitz constant $L$ is assumed to be
known. However, $L$ might be difficult to get, and even if $L$ is given, local
Lipschitz constants may be much smaller than $L$ such that the step size
$\frac{1}{L}$ becomes too conservative. A widely adopted solution is
backtracking linesearch, where the step size is adaptively chosen. Tseng
\cite{tseng2008accelerated} presented a sufficient condition on the step size to
preserve the convergence rate of $\mathcal{N}_L$. Becker et al.\ \cite[\S
5.3]{becker2010templates} proposed an alternative condition that is numerically
more stable to verify, and they also discussed implementation issues. Gonzaga
and Karas \cite{gonzaga2008optimal} developed a linesearch scheme that preserves
the convergence rate of $\mathcal{N}_{\mu,L}$ when only $\mu$ is
given. Linesearch schemes generally do not need explicit knowledge of $L$, but a
single search may require evaluating the objective function for several
times. Hence, even if $L$ is provided, it is still problem-dependent whether we
should use the constant step $\mathcal{N}_L$/$\mathcal{N}_{\mu,L}$ or a
backtracking linesearch.

On strongly convex functions with Lipschitz gradient, $\mathcal{N}_L$ may
converge at a rate $\mathcal{O}(1/k^2)$ while even the steepest gradient descent
method has linear convergence. Note that the optimal method
$\mathcal{N}_{\mu,L}$ takes the same form as $\mathcal{N}_L$. The only
difference is the acceleration parameter. $\mathcal{N}_L$ increases the
acceleration parameter gradually. $\mathcal{N}_{\mu,L}$, given the global
convexity parameter $\mu$, sets the acceleration parameter to a constant that
guarantees linear convergence at an optimal rate. However, $\mu$ is not always
known. Nesterov \cite{nesterov2007gradient} proposed a practical approach to
discover strong convexity: restarting $\mathcal{N}_L$ after a certain number of
iterations. Theoretically, whether we should restart $\mathcal{N}_L$ depends on
the local condition number. Empirically, even with sub-optimal choices, linear
convergence rate can be achieved. See Becker et al.~\cite[\S
5.6]{becker2010templates} for more details. Gonzaga and Karas
\cite{gonzaga2008optimal} developed an adaptive procedure to estimate $\mu$ at
the cost of function evaluations.

In this work, we assume that both $\mu$ and $L$ are given and only the gradient
function is used to maintain minimal cost per iteration. We save gradient calls
based on the global trend of $\| f'(y_k) \|_2$. We argue that there are many
cases where $\mu$ and $L$ are easy to obtain. $L$ can be easily estimated for a
quadratic function, or derived from a smooth approximation of a non-smooth
function \cite{nesterov2005smooth}, and $\mu$ can be derived from a quadratic
regularization term, e.g., $\frac{\mu}{2}\|x-c\|^2$, or by adding a quadratic
term to the objective manually and then performing sequential updates.


\section{Numerical experiments}
\label{sec:numer-exper}

We compare the four variants of $\mathcal{N}_{\mu,L}^\alpha$ with
$\mathcal{N}_L$ and $\mathcal{N}_{\mu,L}$. We implement
$\mathcal{N}_{\mu,L}^\alpha$ in MATLAB. The source code is available for
download\footnote{\url{http://www.stanford.edu/~mengxr/pub/acc_nesterov.html}}
together with code that can be used to reproduce our results. $\mathcal{N}_L$
doesn't take $\mu$ as input and converges with rate $\mathcal{O}(1/k^2)$. To
recover linear convergence, as suggested by Nesterov~\cite{nesterov2007gradient}
and Becker et al.~\cite{becker2010templates}, we restart $\mathcal{N}_L$ after a
certain number of iterations. The optimal number of iterations between restarts
is problem-dependent. For each test, we restart $\mathcal{N}_L$ every 10, 100,
and 1000 iterations respectively, compare the convergence rates with
$\mathcal{N}_L$ without restart, and present the best result. The experiments
were performed on a laptop that has two Intel Core Duo CPU cores at clock rate
2.0GHz and 4GB RAM. Only one core was used to remove the effect of
multi-threading. We compare the convergence based on number of gradient calls
and on running times, rather than on number of iterations, because Nesterov's
methods call the gradient function exactly once per iteration, but
$\mathcal{N}_{\mu,L}^\alpha$ may call the gradient function twice per
iteration. The running times were measured in wall-clock times.

\subsection{Ridge regression}

\begin{figure}
  \centering
  \includegraphics[width=0.45\textwidth, clip, trim=1.5cm 7cm 2.5cm
7cm]{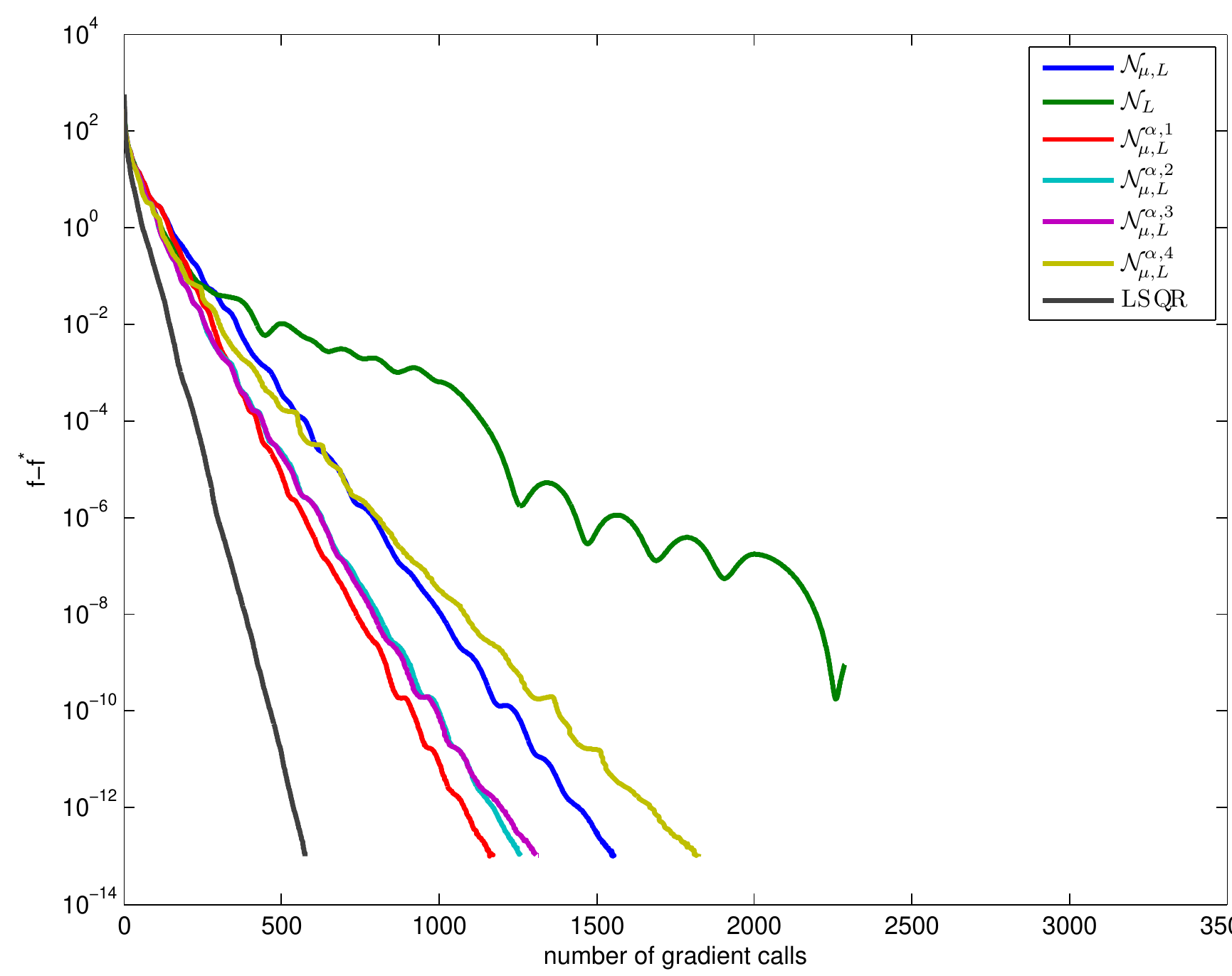}
  \includegraphics[width=0.45\textwidth, clip, trim=1.5cm 7cm 2.5cm
7cm]{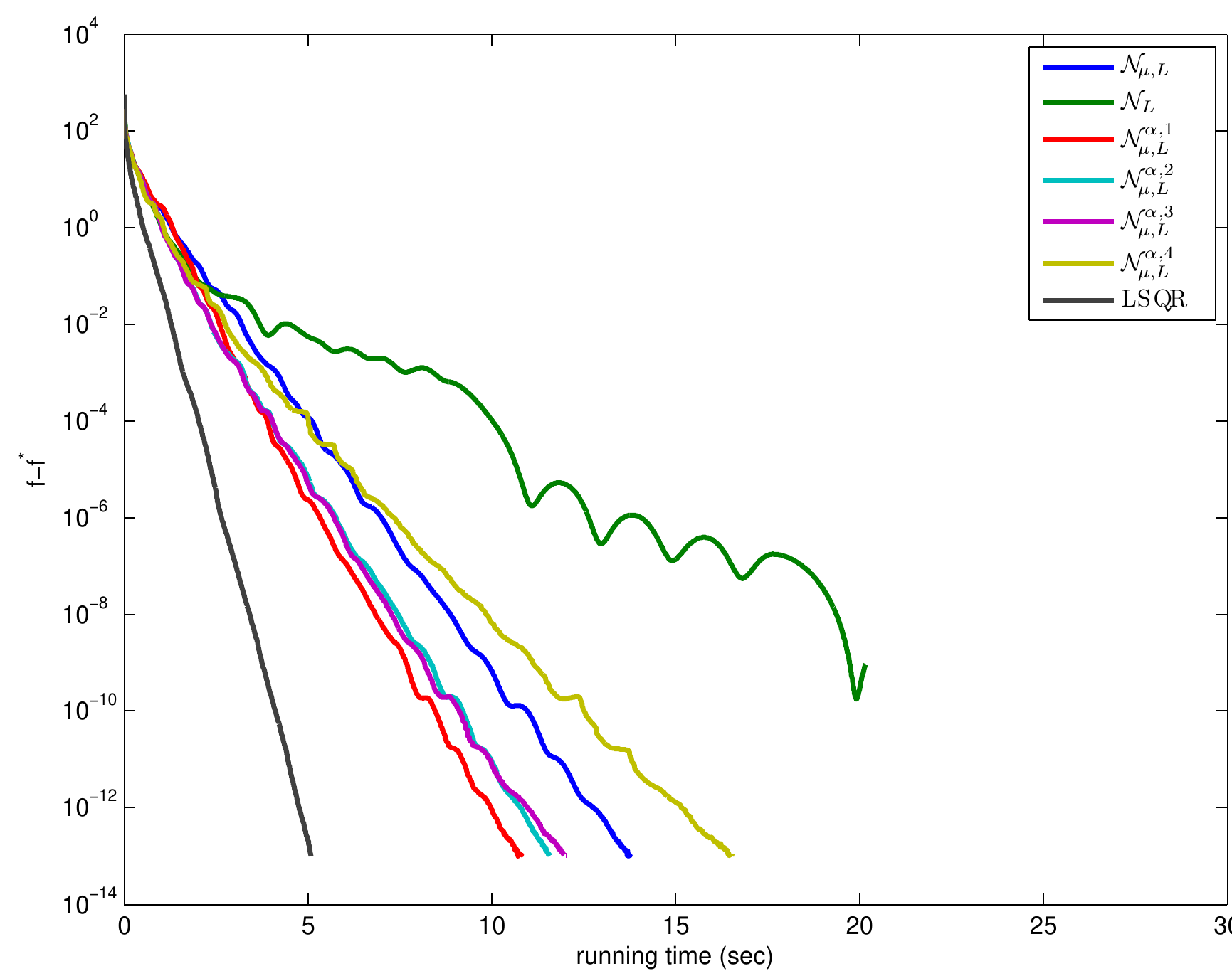}
  \includegraphics[width=0.45\textwidth, clip, trim=1.5cm 7cm 2.5cm
7cm]{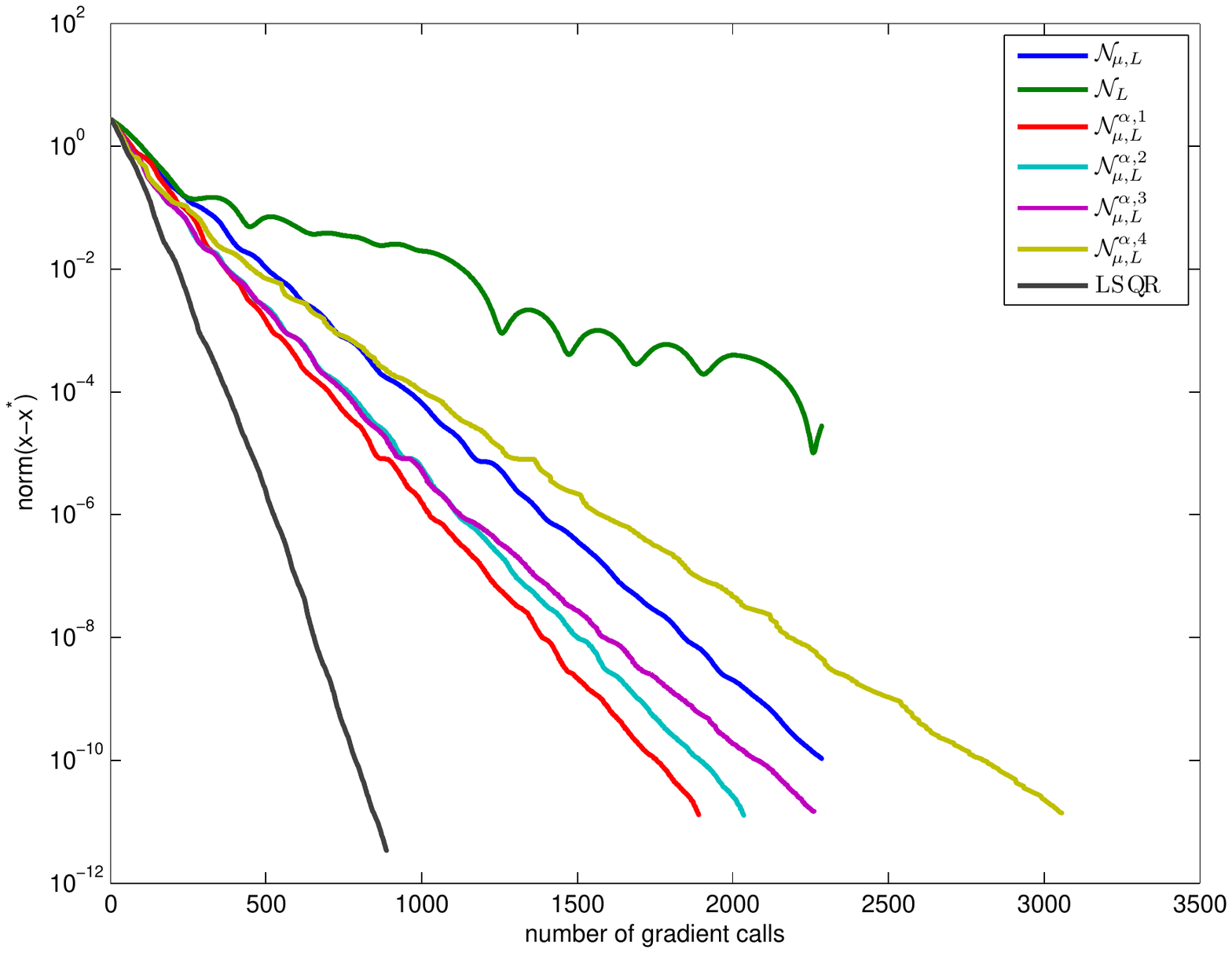}  
  \includegraphics[width=0.45\textwidth, clip, trim=1.5cm 7cm 2.5cm
7cm]{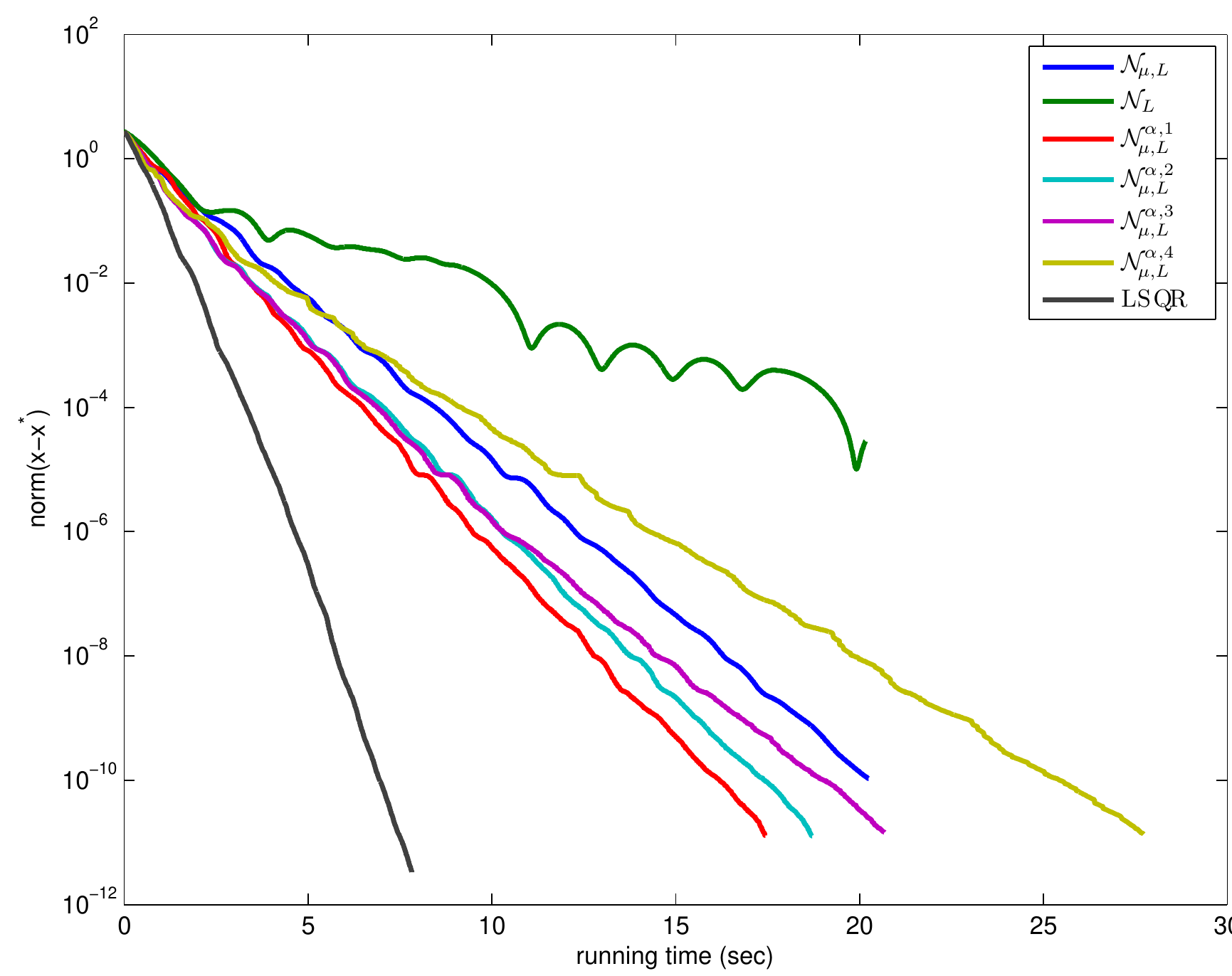}  
\caption{On a ridge regression problem. Top left: $f-f^*$ vs.\ number of
  gradient calls. Top right: $f-f^*$ vs.\ running time. Bottom left:
  $\|x-x^*\|_2$ vs.\ number of gradient calls. Bottom right: $\|x-x^*\|_2$ vs.\
  running time. In terms of convergence speed, we have $\text{LSQR} >
  \mathcal{N}_{\mu,L}^{\alpha,1} > \mathcal{N}_{\mu,L}^{\alpha,2} >
  \mathcal{N}_{\mu,L}^{\alpha,3} > \mathcal{N}_{\mu,L} >
  \mathcal{N}_{\mu,L}^{\alpha,4} >
  \mathcal{N}_L$. $\mathcal{N}_{\mu,L}^{\alpha,4}$ is too aggressive and should
  be used with caution.}
  \label{fig:ridge}
\end{figure}

Our first test is on a ridge regression problem, i.e., a linear least squares
problem with Tikhonov regularization:
\begin{equation*}
  \text{minimize} \quad f(x) = \frac{1}{2} \| A x - b \|_2^2 
  + \frac{\lambda}{2} \| x \|_2^2,
\end{equation*}
where $A \in \mathbb{R}^{m \times n}$ is the measurement matrix, $b \in
\mathbb{R}^m$ is the response vector, and $\lambda > 0$ is the ridge
parameter. The unique solution is given by $x^* = ( A^T A + \lambda I )^{-1} A^T
b$.

$f(x)$ is a positive definite quadratic function, the simplest function type in
the $S_{\mu,L}$ family. $f(x)$ has Lipschitz gradient with constant $L = \| A
\|_2^2 + \lambda$ and strong convexity with parameter $\mu = \lambda$. It is
easy to show that $\mathcal{N}_{\mu,L}$ automatically achieves better
convergence rate on positive definite quadratic functions by exploring the
eigenspace. We have
\begin{equation*} \| x_k - x^* \|_2 \leq C_0 \left( 1 - \sqrt{\rho} \right)^k \|
x_0 - x^* \|_2
\end{equation*} for some constant $C_0 > 0$ and hence
\begin{equation*} f(x_k) - f^* \leq \frac{L}{2} \| x_k - x^* \|_2^2 \leq
\frac{C_0^2 L}{2} ( 1 - \sqrt{\rho} )^{2k} \| x_0 - x^* \|_2^2.
\end{equation*}
We omit the proof because it is purely mechanic work. Another important fact
about positive definite quadratic functions is that there exist algorithms that
can achieve the lower complexity bound derived by Nesterov~\cite[p.\
68]{nesterov2004introductory}, e.g., the conjugate gradient (CG) method. We
refer readers to Luenberger~\cite{luenberger1973introduction} for a detailed
analysis of CG's convergence rate. For least squares problems, LSQR
\cite{paige1982lsqr} is preferable because LSQR is equivalent to applying CG to
the normal equation in exact arithmetic but numerically more stable. The purpose
of this test is not to compete with LSQR, which is specifically designed to
solve least squares problems, but to treat LSQR as an ideal method and see how
$\mathcal{N}_{\mu,L}^\alpha$ can reduce the gap between $\mathcal{N}_{\mu,L}$
and the ideal method on the simplest function family in $S_{\mu,L}$.

We choose $m = 1200$, $n = 2000$, and $\lambda = 1.0$.  We generate $A$ from $U
\Sigma V^T$ where $U \in \mathbb{R}^{m \times m}$ and $V \in \mathbb{R}^{n
  \times m}$ are orthonormal matrices chosen at random, $\Sigma \in
\mathbb{R}^{m \times m}$ is a diagonal matrix with diagonal elements linearly
spaced between and including $100$ and $1$. $b = \text{randn}(m,1)$ is a random
vector whose entries are i.i.d.\ samples drawn from the standard normal
distribution. Although the exact value is known, $\|A\|_2^2$ is estimated by
applying the power method to $AA^T$. We have $\mu = 1$ and $L \approx
10001$. Figure \ref{fig:ridge} shows the comparison results. LSQR leads as
expected. $\mathcal{N}_{\mu,L}^{\alpha,1}$, $\mathcal{N}_{\mu,L}^{\alpha,2}$,
and $\mathcal{N}_{\mu,L}^{\alpha,3}$ form the second group with
$\mathcal{N}_{\mu,L}^{\alpha,1}$ having a slight
edge. $\mathcal{N}_{\mu,L}^{\alpha,4}$ falls behind all other variants of
$\mathcal{N}_{\mu,L}^\alpha$ and $\mathcal{N}_{\mu,L}$ because it is too
aggressive on choosing an $\alpha_k$ and falls back to $\alpha_k = \sqrt{\rho}$
frequently. Hence $\mathcal{N}_{\mu,L}^{\alpha,4}$ should be used with
caution. $\mathcal{N}_L$, even with restart, is the slowest among competitive
methods. We see $\mathcal{N}_{\mu,L}^{\alpha,1}$ approximately reduces the gap
between $\mathcal{N}_{\mu,L}$ and LSQR by a factor of $30\%$ in terms of number
of gradient calls.

\subsection*{Anisotropic bowl}

\begin{figure} \centering
  \includegraphics[width=0.45\textwidth, clip, trim=1.5cm 7cm 2.5cm
7cm]{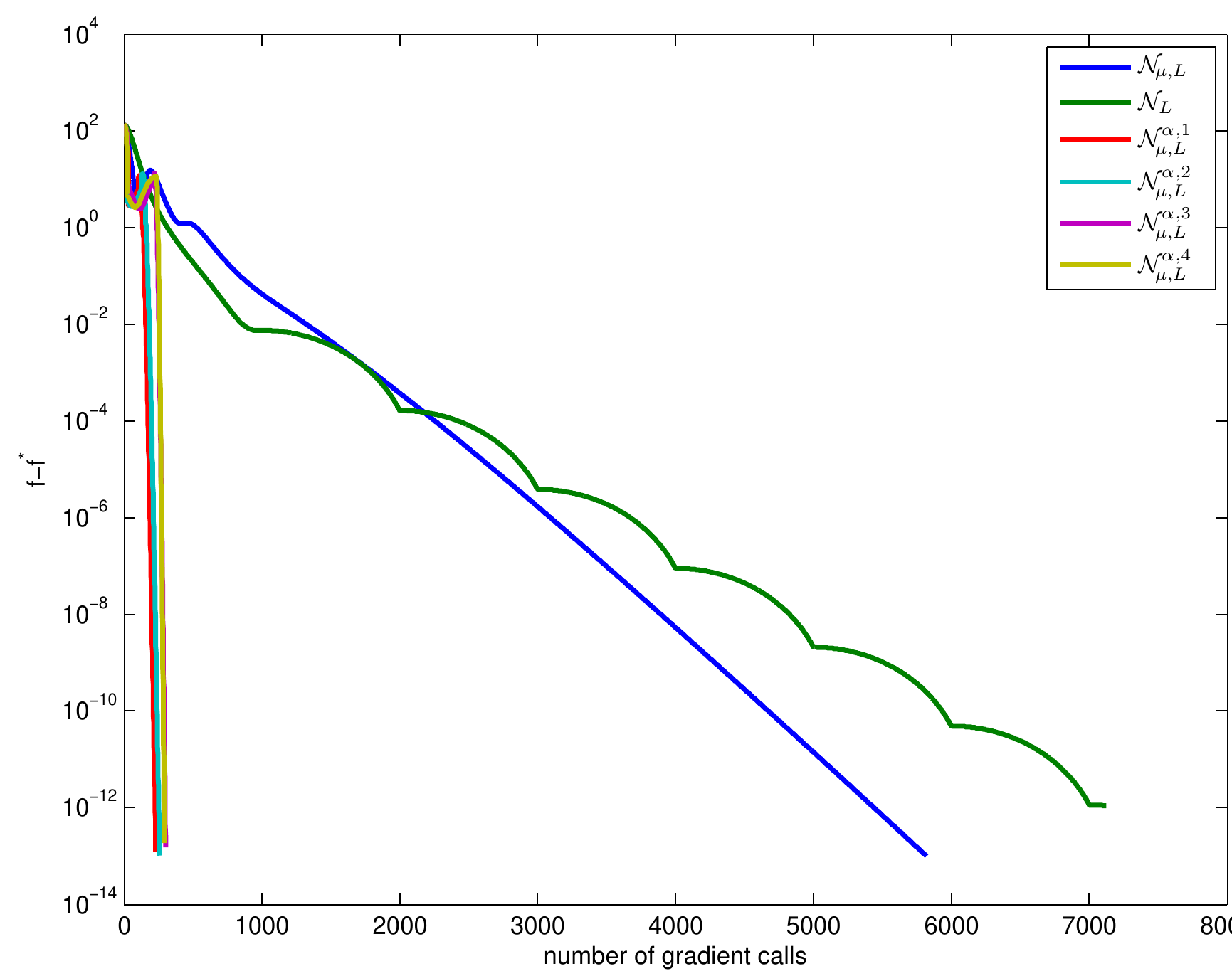}
  \includegraphics[width=0.45\textwidth, clip, trim=1.5cm 7cm 2.5cm
7cm]{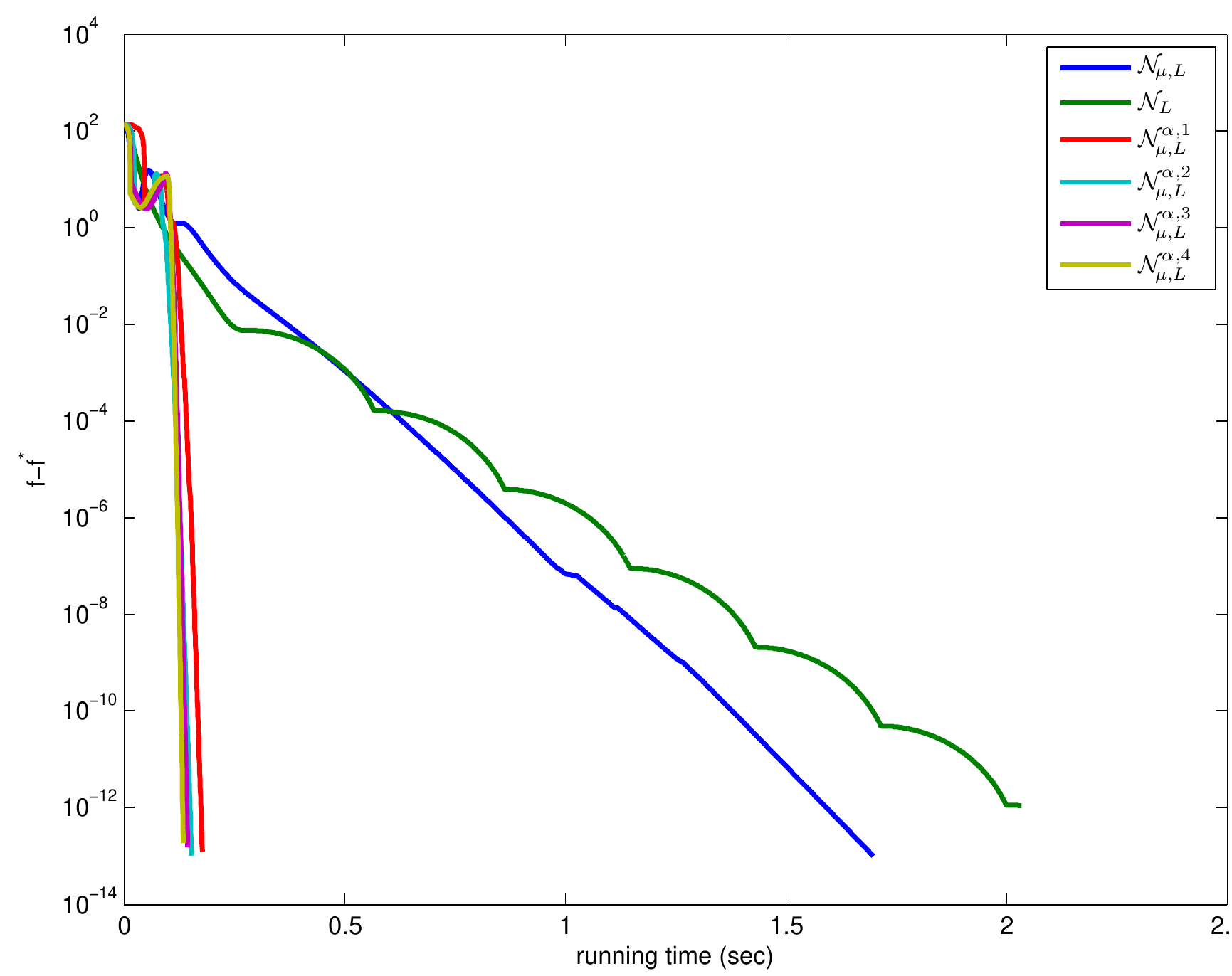}
  \includegraphics[width=0.45\textwidth, clip, trim=1.5cm 7cm 2.5cm
7cm]{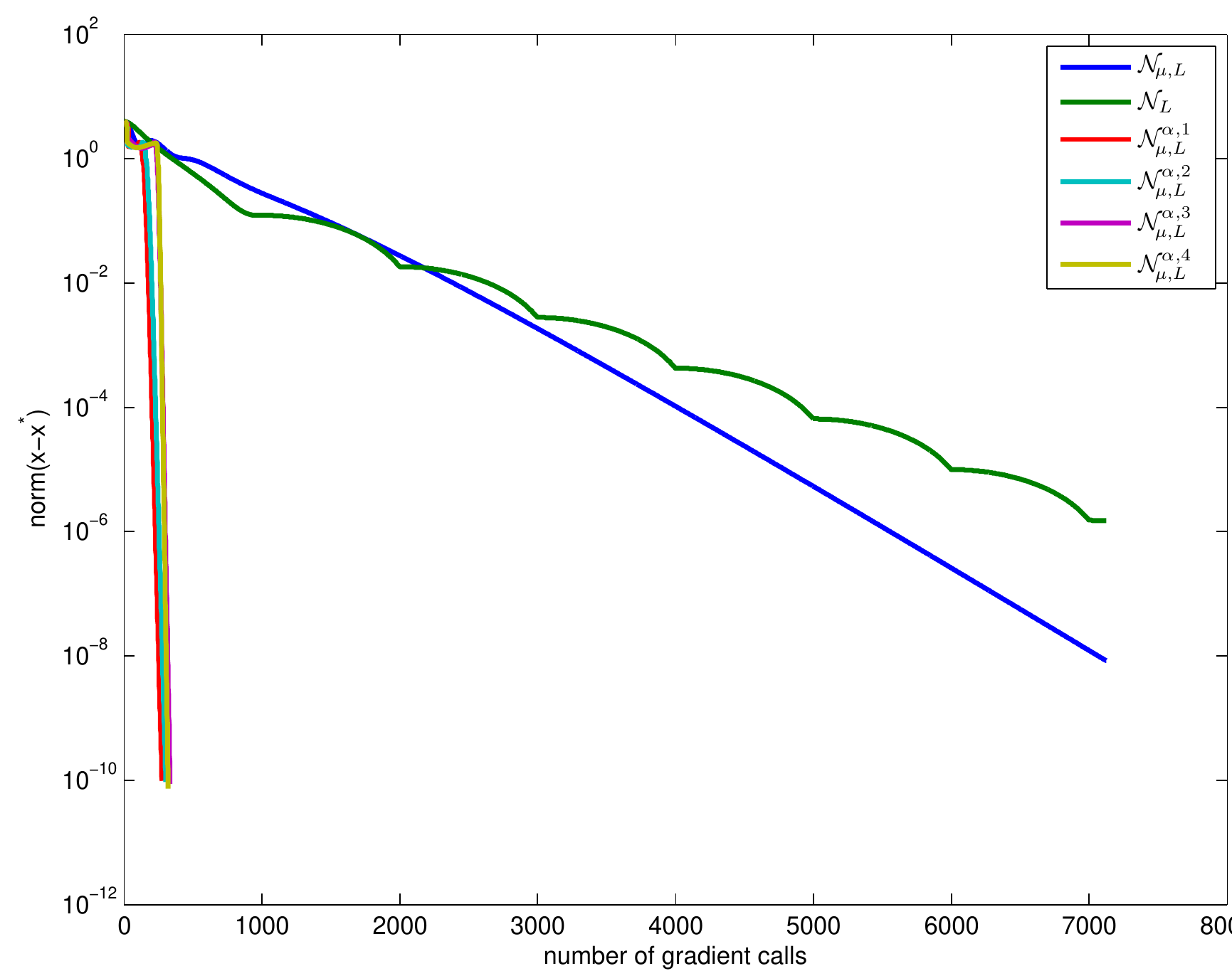}  
  \includegraphics[width=0.45\textwidth, clip, trim=1.5cm 7cm 2.5cm
7cm]{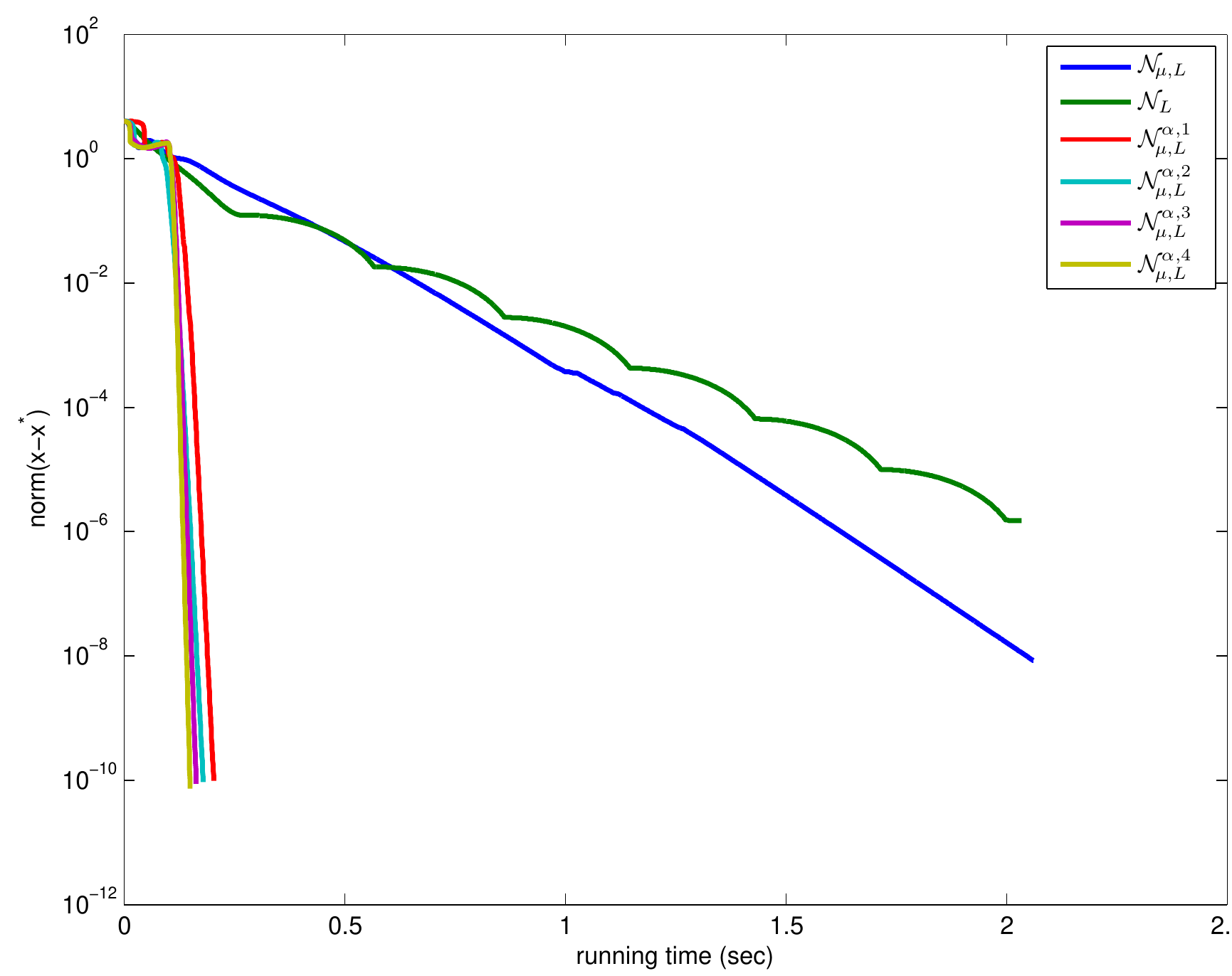} 
\caption{On an anisotropic bowl. Top left: $f-f^*$ vs.\ number of gradient
  calls. Top right: $f-f^*$ vs.\ running time. Bottom left: $\|x-x^*\|_2$ vs.\
  number of gradient calls. Bottom right: $\|x-x^*\|_2$ vs.\ running time. All
  the variants of $\mathcal{N}_{\mu,L}^\alpha$ converge significantly faster
  than $\mathcal{N}_{\mu,L}$ or $\mathcal{N}_L$.}
  \label{fig:ani_bowl}
\end{figure} 

The second test is on a bowl-shaped function, which is anisotropic along
different directions:
\begin{eqnarray*} \text{minimize} &\quad& f(x) = \sum_{i=1}^n i \cdot x_{(i)}^4
+ \frac{1}{2} \|x\|_2^2 \\ \text{subject to} && \|x\|_2 \leq \tau,
\end{eqnarray*} 
where we use $x_{(i)}$ to indicate the $i$-th element of $x$. We put a
constraint to make $f(x)$ have a Lipschitz continuous gradient over the feasible
region. If $x_k$ falls outside the feasible region, we project it back to the
nearest feasible point. By doing so, we know the function value will be
decreased, so the convergence result still holds. We use this example to test
the performance of $\mathcal{N}_{\mu,L}^\alpha$ and competitive methods when the
gradient has local Lipschitz constants that are much smaller than the global
one.

We choose $n = 500$, $\tau = 4$, and $x_0 = \frac{\tau}{\sqrt{n}}
\mathbf{1}$. With these choices, we have $L = 12 n \tau^2+1 = 96001$ and $\mu =
1$. Figure \ref{fig:ani_bowl} draws the convergence results. We see that all the
variants of $\mathcal{N}_{\mu,L}^\alpha$ converge significantly faster than
$\mathcal{N}_{\mu,L}$ or $\mathcal{N}_L$. For example, to reach $f(x_k) - f^* <
10^{-12}$, variants of $\mathcal{N}_{\mu,L}^{\alpha,1}$ take about $200$
gradient calls, $\mathcal{N}_{\mu,L}$ takes $5500$ gradient calls, and
$\mathcal{N}_L$ takes $7000$ gradient calls. The differences among the four
variants of $\mathcal{N}_{\mu,L}^\alpha$ are really small.

\begin{figure} \centering
  \includegraphics[width=0.45\textwidth, clip, trim=1.5cm 7cm 2.5cm
  7cm]{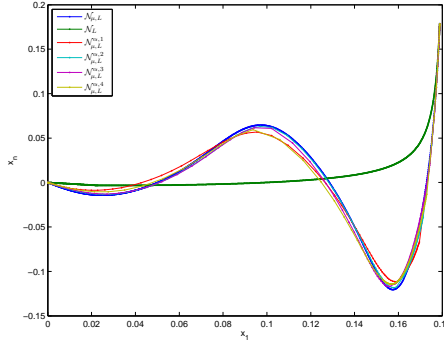}
  \caption{Projected trajectory of $\{x_k\}$ on the plane spanned by $x_{(1)}$
    and $x_{(n)}$. $\mathcal{N}_{\mu,L}$ and $\mathcal{N}_{\mu,L}^\alpha$ are
    almost following the same path, though $\mathcal{N}_{\mu,L}$ makes little
    progress per step, while $\mathcal{N}_{\mu,L}^\alpha$ has many large
    steps. We say $\mathcal{N}_{\mu,L}^\alpha$ accelerates $\mathcal{N}_{\mu,L}$
    in this sense.}
  \label{fig:traj}
\end{figure}

To investigate further, we plot the projected trajectory of $\{x_k\}$ on the
plane spanned by $x_{(1)}$ and $x_{(n)}$ for each method. In Figure
\ref{fig:traj} we see that the point sequences generated by
$\mathcal{N}_{\mu,L}^{\alpha}$ and $\mathcal{N}_{\mu,L}$ are almost following
the same path. However, $\mathcal{N}_{\mu,L}$ makes very little progress per
step, while $\mathcal{N}_{\mu,L}^\alpha$ jumps along the path. In this sense we
say $\mathcal{N}_{\mu,L}^\alpha$ is indeed accelerating $\mathcal{N}_{\mu,L}$.

\subsection*{Smooth-BPDN} 

The third test is on a smoothed and strongly convex version of the basis pursuit
denoising (BPDN) problem of Chen et al.~\cite{chen2001atomic}:
\begin{equation*} 
  \text{minimize} \quad f(x) = \frac{1}{2} \| A x - b \|_2^2 +
  \lambda \| x \|_{\ell_1,\tau} + \frac{\rho}{2} \|x\|_2^2,
\end{equation*} 
where $\| \cdot \|_{\ell_1,\tau}$ is given by
\begin{equation*} \| x \|_{\ell_1,\tau} =
  \begin{cases} |x| - \frac{\tau}{2} & \text{if } |x| \geq \tau \\
    \frac{1}{2 \tau} x^2 & \text{if } |x| < \tau
  \end{cases}
\end{equation*} 
if $x$ is a scalar and $\| x \|_{\ell_1,\tau} = \sum_{i=1}^n \| x_{(i)}
\|_{\ell_1,\tau}$ if $x$ is a vector in $\mathbb{R}^n$. $\| \cdot
\|_{\ell_1,\tau}$ is a smoothed version of the $\ell_1$ norm, also recognized as
the Huber penalty function with half-width $\tau$. $\lambda > 0$ and $\rho > 0$
are parameters controlling the penalty terms. The quadratic term $\frac{\rho}{2}
\| x \|_2^2$ makes the function strongly convex. $f(x)$ has Lipschitz gradient
with constant $L = \|A\|_2^2 + \frac{\lambda}{\tau} + \rho$ and strong convexity
with parameter $\mu = \rho$.

We set $A = \frac{1}{\sqrt{n}} \cdot \text{randn}(m,n)$, where $m = 800$ and $n
= 2000$, $\lambda = 0.05$, $\tau = 0.0001$, and $\mu = 0.05$. The true signal is
a random sparse vector with $40$ nonzeros. $b = A x^* + e$, where $e = 0.01
\frac{\|b\|_2}{\sqrt{m}} \cdot \text{randn}(m,1)$ is a Gaussian noise.
$\|A\|_2^2$ is estimated by applying the power method to $A A^T$. The value is
around $1.63$. Hence we have
\begin{equation*}
  L = \|A\|_2^2 + \frac{\lambda}{\tau} + \rho \approx 502.7  \text{ and } \mu = 0.05.
\end{equation*}
There is no analytic solution for this problem. We apply $\mathcal{N}_{\mu,L}$
to the problem with a small tolerance on the gradient norm and use the
approximate solution returned by $\mathcal{N}_{\mu,L}$ as the optimal
solution. Figure \ref{fig:bpdn} presents the results. All variants of
$\mathcal{N}_{\mu,L}^\alpha$ run faster than $\mathcal{N}_{\mu,L}$ or
$\mathcal{N}_L$. It takes about $750$ gradient calls for
$\mathcal{N}_{\mu,L}^\alpha$ to reach $f(x_k) - f^* < 10^{-12}$, $1300$ for
$\mathcal{N}_{\mu,L}$, and $1900$ for $\mathcal{N}_L$. The corresponding running
times are around $5$, $7.5$, and $11.5$ seconds, respectively.
$\mathcal{N}_{\mu,L}^{\alpha,4}$ is slow at the beginning but becomes the
fastest method at the end. However, the differences among the four variants of
$\mathcal{N}_{\mu,L}^\alpha$ are not big.

\begin{figure} \centering
  \includegraphics[width=0.45\textwidth, clip, trim=1.5cm 7cm 2.5cm
7cm]{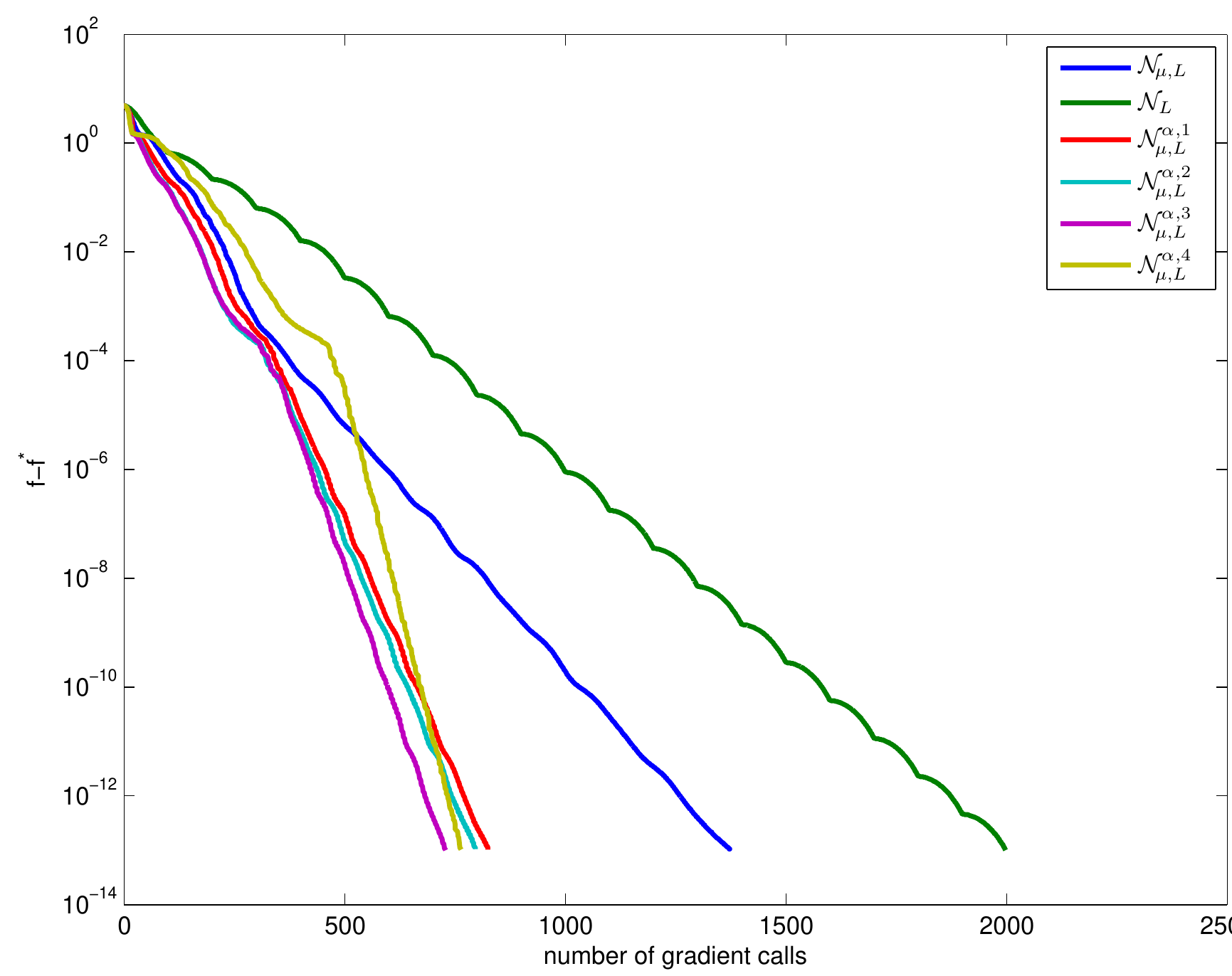}
  \includegraphics[width=0.45\textwidth, clip, trim=1.5cm 7cm 2.5cm
7cm]{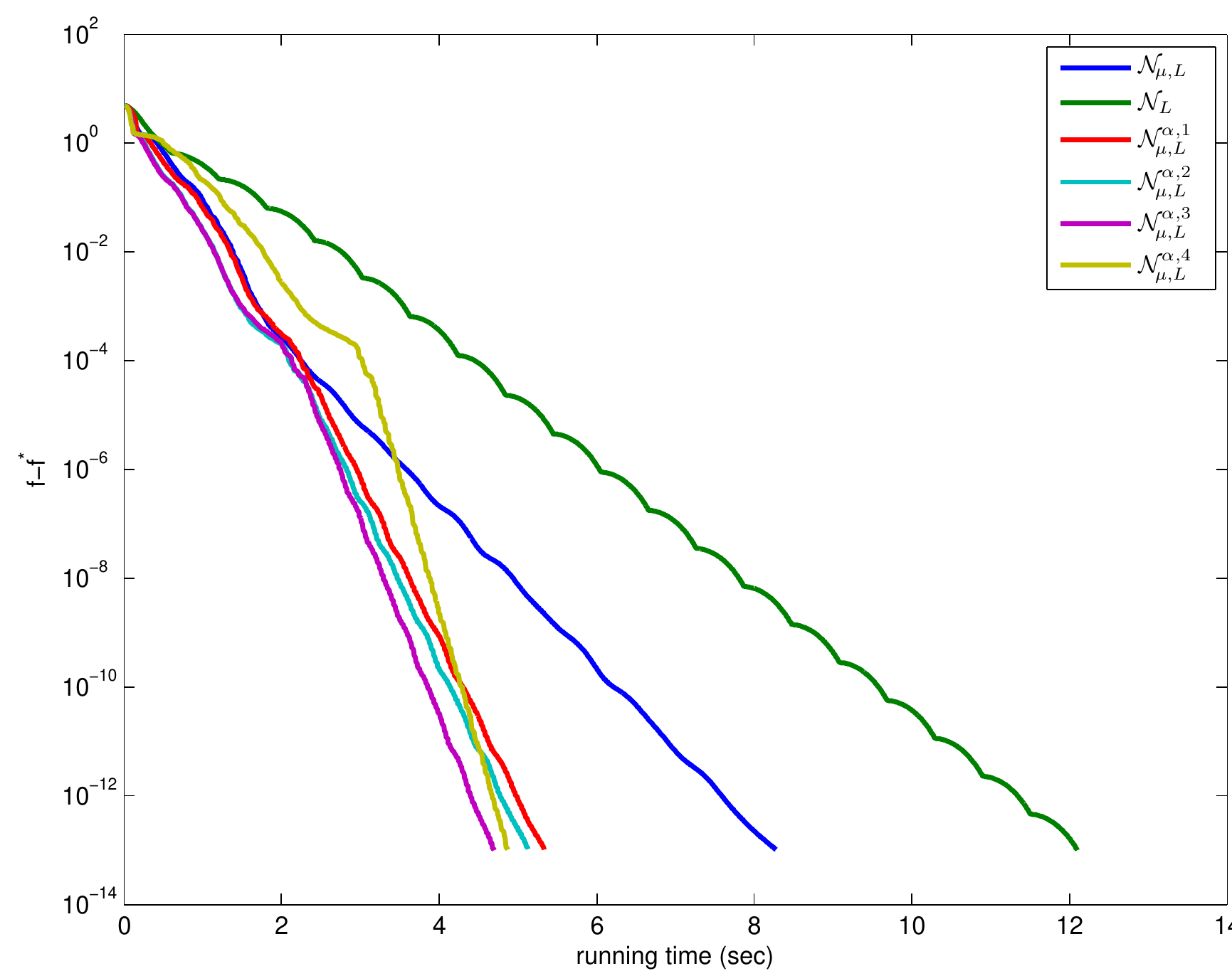}
  \includegraphics[width=0.45\textwidth, clip, trim=1.5cm 7cm 2.5cm
7cm]{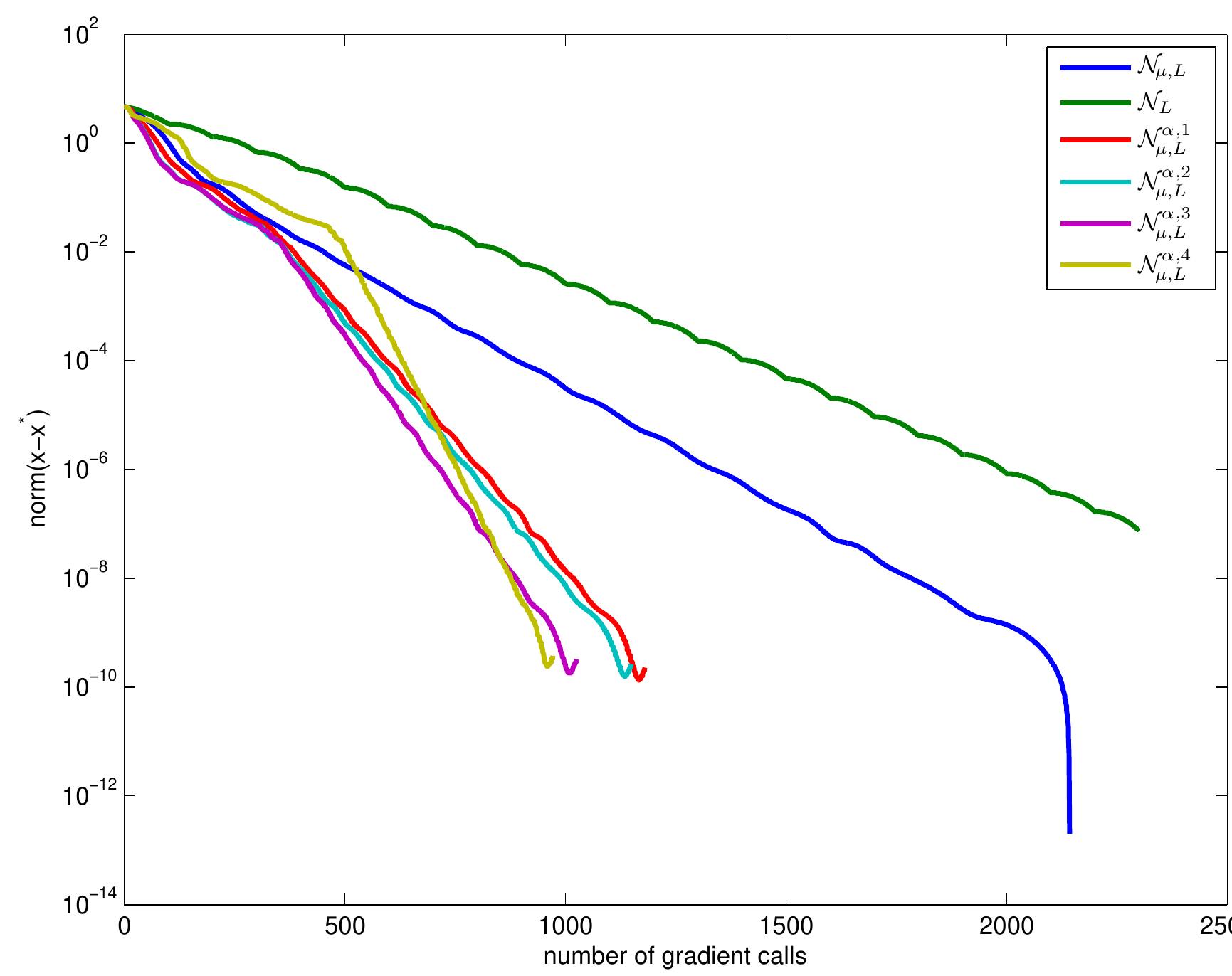}  
  \includegraphics[width=0.45\textwidth, clip, trim=1.5cm 7cm 2.5cm
7cm]{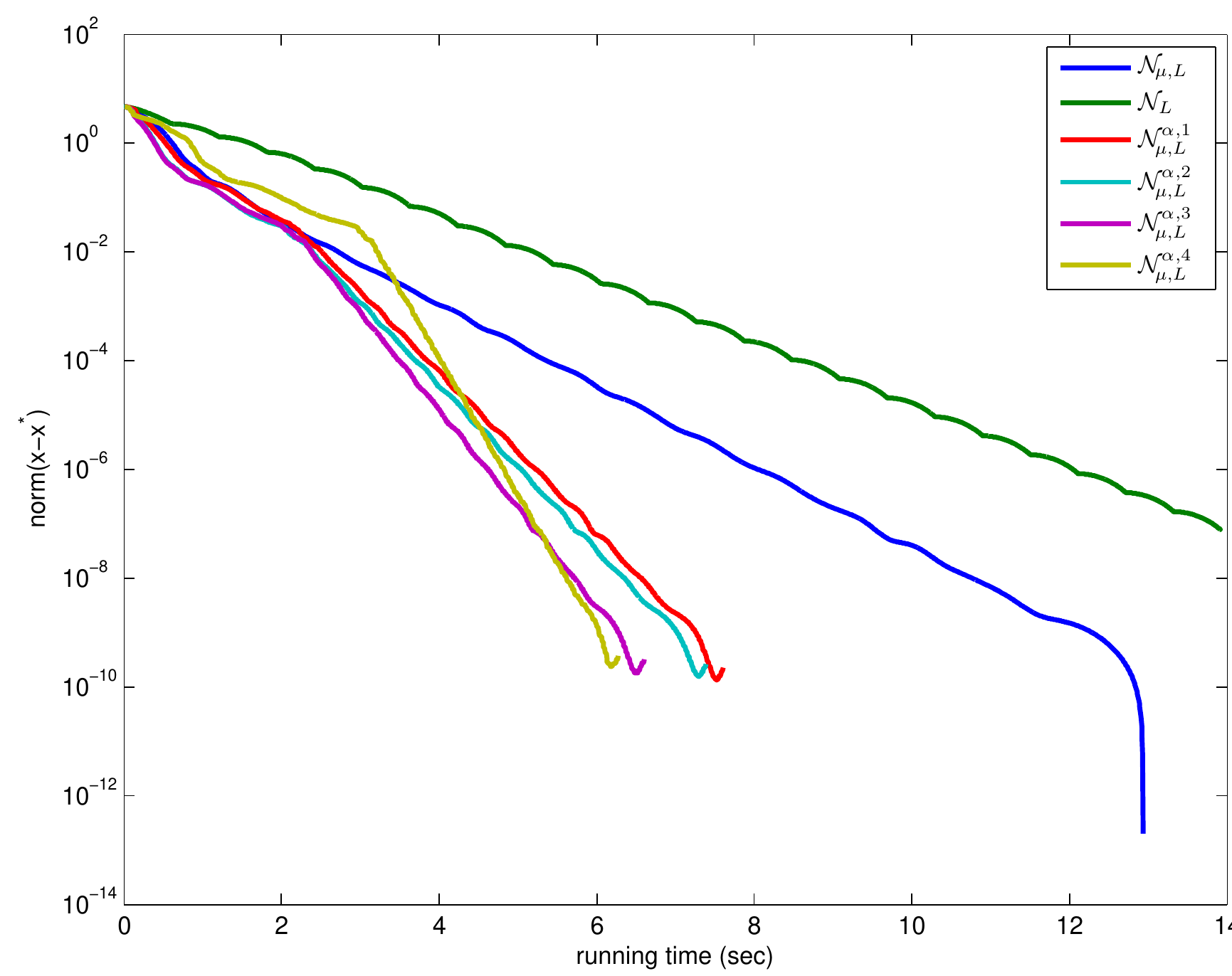} 
\caption{Smooth-BPDN. Top left: $f-f^*$ vs.\ number of gradient calls. Top
  right: $f-f^*$ vs.\ running time. Bottom left: $\|x-x^*\|_2$ vs.\ number of
  gradient calls. Bottom right: $\|x-x^*\|_2$ vs.\ running time. All the
  variants of $\mathcal{N}_{\mu,L}^\alpha$ converge faster than
  $\mathcal{N}_{\mu,L}$ or $\mathcal{N}_L$.}
  \label{fig:bpdn}
\end{figure}

\begin{figure} \centering
  \includegraphics[width=0.5\textwidth, clip, trim=1.5cm 7cm 2.5cm
  7cm]{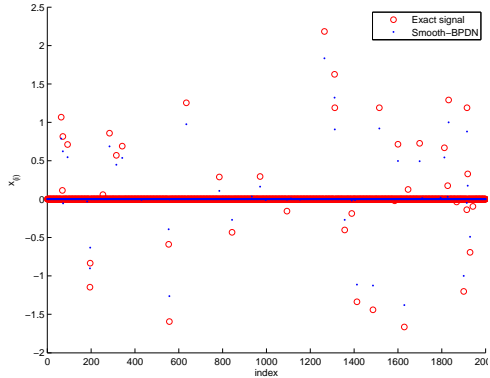}
  \caption{Exact signal vs.\ smooth-BPDN solution. The smooth-BPDN solution is
    very close to a soft-thresholded version of the exact signal. Smooth-BPDN
    solution recovers all the coefficients with large magnitude.}
  \label{fig:sig}
\end{figure}

Though the purpose of this test is not to recover sparse signals but to compare
$N_{\mu,L}^\alpha$ with competitive methods, we show that smooth-BPDN does
recover sparse signals and hence it has practical value as well. Figure
\ref{fig:sig} compares the smooth-BPDN solution with the exact signal. We see
the smooth-BPDN solution is very similar to a soft-thresholded version of the
exact signal. It recovers all the coefficients with large magnitude.

In summary, the proposed method $\mathcal{N}_{\mu,L}^\alpha$ can effectively
accelerate Nesterov's method $\mathcal{N}_{\mu,L}$ in all the tests we
present. Among the four variants, the first, second, and the third perform quite
similarly. The fourth, the most aggressive one, may fall back frequently, as we
see in the ridge regression case. Though it is the fastest method in the
smooth-BPDN test, we don't recommend it in general. Since the first heuristic is
the most conservative one and delivers comparable performance in all the three
tests, we suggest using $\mathcal{N}_{\mu,L}^{\alpha,1}$ as the default setting.


\section{Conclusion and future work}
\label{sec:conclusion}

We modified Nesterov's constant step gradient method for strongly convex
functions with Lipschitz gradient such that, at each iteration, we try to choose
an $\alpha_k > \sqrt{\rho}$ adaptively while preserving the estimate sequence,
where $\alpha_k$ controls the rate of decrease. $\mathcal{N}_{\mu,L}^\alpha$,
the modified method, has at least the same convergence speed as Nesterov's
method. Though it may evaluate the gradient function twice per iteration, in
practice it effectively accelerates the speed of convergence for many
problems. We propose four heuristics for choosing $\alpha_k$, compare their
performance in the numerical experiments, and suggest a default one to use.

Note that we don't utilize all the degrees of freedom in constructing our
method. The sequences $\{y_k\}$ and $\{x_k\}$ are still following Nesterov's, so
that we can reduce the number of calls to the gradient function. However,
further exploration on the choices of $\{y_k\}$, $\{x_k\}$, and $\{\alpha_k\}$
may help discover more efficient methods or help design variable step size
methods. We leave those possible directions as our future work.

The authors would like to thank Michael A.\ Saunders for useful comments on a
previous draft of this paper.


\bibliographystyle{siam}
\bibliography{acc_nesterov}

\end{document}